\documentclass[reqno,oneside,11pt]{amsart}

\usepackage[T1]{fontenc}
\usepackage{times,mathptm}
\usepackage{amssymb,epsfig,verbatim,mathtools,float,hyperref,xcolor}
\usepackage{pgf,tikz}
\usepackage[all,cmtip]{xy}
\theoremstyle{plain}

\newtheorem{thm}{Theorem}[section]
\newtheorem{cor}[thm]{Corollary}

\newtheorem{pro}[thm]{Proposition}
\newtheorem{lem}[thm]{Lemma}

\newtheorem{proposition-principale}[thm]{Proposition principale}
\newtheorem{theoalph}{Theorem}

\newtheorem{coralph}[theoalph]{Corollary}

\theoremstyle{definition}

\addtocounter{section}{0}              
\numberwithin{equation}{section}       
\begin{document}

\setlength{\baselineskip}{0.54cm}         
\title[Dynamical number of base-points of non base-wandering Jonqui\`eres twists]{Dynamical number of base-points of \\
non base-wandering Jonqui\`eres twists}
\date{}

\author{Julie D\'eserti}\thanks{The 
 author was partially supported by the ANR grant Fatou ANR-$17$-CE$40$-$0002$-$01$ and the ANR grant Foliage ANR-$16$-CE$40$-$0008$-$01$.}
\address{Universit\'e d'Orl\'eans, Institut Denis Poisson, route de Chartres, $45067$ Orl\'eans Cedex $2$, France}
\email{deserti@math.cnrs.fr}

\subjclass[2010]{}


\begin{abstract} 
We give some properties of the dynamical number 
of base-points of birational self-maps of 
$\mathbb{P}^2_\mathbb{C}$. 

In particular we give a formula to determine the 
dynamical number of base-points of non 
base-wandering Jonqui\`eres twists. 
\end{abstract}

\maketitle

\section{Introduction}

The \textsl{plane Cremona group} $\mathrm{Bir}(\mathbb{P}^2_\mathbb{C})$ 
is the group of birational maps of the complex projective
plane $\mathbb{P}^2_\mathbb{C}$. It is isomorphic to the 
group of $\mathbb{C}$-algebra automorphisms of $\mathbb{C}(X,Y)$, 
the function field of $\mathbb{P}^2_\mathbb{C}$. Using a 
system of homogeneous coordinates $(x:y:z)$ a birational map
$f\in\mathrm{Bir}(\mathbb{P}^2_\mathbb{C})$ can be written as
\[
(x:y:z)\dashrightarrow(P_0(x,y,z):P_1(x,y,z):P_2(x,y,z))
\]
where $P_0$, $P_1$ and $P_2$ are homogeneous polynomials
of the same degree without common factor. This degree does
not depend on the system of homogeneous coordinates. We call it 
the \textsl{degree} of $f$ and denote it by $\deg(f)$. Geometrically
it is the degree of the pull-back by~$f$ of a general 
projective line. Birational maps of degree $1$ are homographies
and form the group 
$\mathrm{Aut}(\mathbb{P}^2_\mathbb{C})=\mathrm{PGL}(3,\mathbb{C})$
of automorphisms of the projective plane.

\medskip

\begin{itemize}
\item[$\diamond$] \textbf{Four types of elements}. 

The elements $f\in\mathrm{Bir}(\mathbb{P}^2_\mathbb{C})$
can be classified into exactly one of the four following 
types according to the growth of the sequence 
$(\deg(f^k))_{k\in\mathbb{N}}$ (\emph{see} 
\cite{DillerFavre, BlancDeserti}):
\begin{enumerate}
\item The sequence $(\deg(f^k))_{k\in\mathbb{N}}$ is 
bounded, $f$ is either of finite order or conjugate
to an automorphism of $\mathbb{P}^2_\mathbb{C}$;
we say that $f$ is an \textsl{elliptic element}.

\item The sequence $(\deg(f^k))_{k\in\mathbb{N}}$
grows linearly, $f$ preserves a unique pencil of 
rational curves and $f$ is not conjugate to an 
automorphism of any rational projective surface; 
we call $f$ a \textsl{Jonqui\`eres twist}.

\item The sequence $(\deg(f^k))_{k\in\mathbb{N}}$ 
grows quadratically, $f$ is conjugate to an 
automorphism of a rational projective surface 
preserving a unique elliptic fibration; we 
call $f$ a \textsl{Halphen twist}.

\item The sequence $(\deg(f^k))_{k\in\mathbb{N}}$
grows exponentially and we say that $f$ is 
\textsl{hyperbolic}.
\end{enumerate}

\medskip

\item[$\diamond$] \textbf{The Jonqui\`eres group}.

Let us fix an affine chart of $\mathbb{P}^2_\mathbb{C}$ 
with coordinates $(x,y)$. The 
\textsl{Jonqui\`eres group} $\mathrm{J}$ is the subgroup of 
the Cremona group of all maps of the form 
\begin{equation}\label{eq:jonq}
(x,y)\dashrightarrow\left(\frac{A(y)x+B(y)}{C(y)x+D(y)},\frac{ay+b}{cy+d}\right)
\end{equation}
where $\left(\begin{array}{cc}
a & b \\
c & d
\end{array}\right)\in\mathrm{PGL}(2,\mathbb{C})$ and $\left(\begin{array}{cc}
A & B \\
C & D
\end{array}\right)\in\mathrm{PGL}(2,\mathbb{C}(y))$. 
The group $\mathrm{J}$ is the group of all birational
maps of $\mathbb{P}^1_\mathbb{C}\times\mathbb{P}^1_\mathbb{C}$
permuting the fibers of the projection onto the 
second factor; it is isomorphic to the semi-direct
product $\mathrm{PGL}(2,\mathbb{C}(y))\rtimes\mathrm{PGL}(2,\mathbb{C})$.

We can check with (\ref{eq:jonq}) that if $f$ belongs to 
$\mathrm{J}$, then $(\deg(f^k))_{k\in\mathbb{N}}$ grows
at most linearly; elements of $\mathrm{J}$ are either 
elliptic or Jonqui\`eres twists. Let us denote by $\mathcal{J}$ the set
of Jonqui\`eres twist:
\[
\mathcal{J}=\big\{f\in\mathrm{Bir}(\mathbb{P}^2_\mathbb{C})\,\vert\,\text{ $f$ Jonqui\`eres twist }\big\}.
\]
A Jonqui\`eres twist is called a \textsl{base-wandering Jonqui\`eres
twist} if its action on the basis of the rational fibration
has infinite order. 
Let us denote by 
$\mathrm{J}_0$ the normal subgroup of $\mathrm{J}$
 that preserves fiberwise the rational fibration, 
 that is the subgroup of those maps of the form 
\[
(x,y)\dashrightarrow\left(\frac{A(y)x+B(y)}{C(y)x+D(y)},y\right).
\]
The group $\mathrm{J}_0$ is isomorphic to 
$\mathrm{PGL}(2,\mathbb{C}(y))$. The group $\mathrm{J}_0$ has 
three maximal (for the inclusion) uncountable abelian 
subgroups
\begin{align*}
& \mathrm{J}_a=\big\{(x+a(y),y)\,\vert\,a\in\mathbb{C}(y)\big\}, && \mathrm{J}_m=\big\{(b(y)x,y)\,\vert\,b\in\mathbb{C}(y)^*\big\}
\end{align*}
and 
\[
\mathrm{J}_F=\left\{(x,y),\,\left(\frac{c(y)x+F(y)}{x+c(y)},y\right)\,\vert\,c\in\mathbb{C}[y]\right\}
\]
where $F$ denotes an element of 
$\mathbb{C}[y]$ that is not 
a square (\cite{Deserti:compositio}).

Let us associate to 
$f=\left(\frac{A(y)x+B(y)}{C(y)x+D(y)},y\right)\in\mathrm{J}_0$
the matrix $M_f=\left(\begin{array}{cc}
A & B\\
C & D
\end{array}\right)$. The \textsl{Baum Bott index} of~$f$ is $\mathrm{BB}(f)=\frac{\big(\mathrm{Tr}(M_f)\big)^2}{\det(M_f)}$ 
(by analogy with the Baum Bott index of a 
foliation) which is well defined in 
$\mathrm{PGL}$ and is invariant by conjugation.
This invariant $\mathrm{BB}$ indicates the 
degree growth:

\begin{pro}[\cite{CerveauDeserti}]\label{pro:bb}
Let $f$ be a Jonqui\`eres twist that preserves fiberwise
the rational fibration. The rational function 
$\mathrm{BB}(f)$ is constant if and only if $f$ is an
elliptic element.
\end{pro}

A direct consequence is the following: 

\begin{cor}
Let $f$ be a non-base wandering Jonqui\`eres twist;
the rational function $\mathrm{BB}(f)$ is constant 
if and only if $f$ is an elliptic element.
\end{cor}

Every $f\in\mathrm{Bir}(\mathbb{P}^2_\mathbb{C})$ admits a resolution
\[
\xymatrix{
&S \ar[dl]_{\pi_1} \ar[dr]^{\pi_2} &\\
\mathbb{P}^2_\mathbb{C}\ar@{-->}[rr]_f &  & \mathbb{P}^2_\mathbb{C} }
\]
where $\pi_1$, $\pi_2$ are sequences of blow-ups. 
The resolution is \textsl{minimal} if and only if
no $(-1)$-curve of~$S$ are contracted by both 
$\pi_1$ and $\pi_2$. Assume that the 
resolution is minimal; the \textsl{base-points} of
$f$ are the points blown-up by $\pi_1$, which 
can be points of $S$ or infinitely near points.
If $f$ belongs to $\mathrm{J}$, then $f$ has one 
base-point $p_0$ of multiplicity $d-1$ and $2d-2$ 
base-points $p_1$, $p_2$, $\ldots$, $p_{2d-2}$ of 
multiplicity $1$. 
Similarly the map $f^{-1}$ has one 
base-point $q_0$ of multiplicity $d-1$ and $2d-2$ 
base-points $q_1$, $q_2$, $\ldots$, $q_{2d-2}$ of 
multiplicity $1$. Let us denote by 
$\mathbf{e}_m\in\mathrm{NS}(S)$ the 
Néron-Severi class of the total transform of $m$ 
under $\pi_j$ (for $1\leq j\leq 2$). The action
of $f$ on $\ell$ and the classes 
$(\mathbf{e}_{p_j})_{0\leq j\leq 2d-2}$ 
is given by: 
\[
\left\{
\begin{array}{lll}
f_\sharp(\ell)=d\ell-(d-1)\mathbf{e}_{q_0}-\displaystyle\sum_{i=1}^{2d-2}\mathbf{e}_{q_i}\\
f_\sharp(\mathbf{e}_{p_0})=(d-1)\ell-(d-2)\mathbf{e}_{q_0}-\displaystyle\sum_{i=1}^{2d-2}\mathbf{e}_{q_i}\\
f_\sharp(\mathbf{e}_{p_i})=\ell-\mathbf{e}_{q_0}-\mathbf{e}_{q_i}\qquad\forall\,1\leq i\leq 2d-2
\end{array}
\right.
\]

\medskip

\item[$\diamond$] \textbf{Dynamical degree}.

Given a birational self-map $f\colon S\dashrightarrow S$ of a 
complex projective surface, its dynamical degree 
$\lambda(f)$ is a positive real number that measures
the complexity of the dynamics of $f$. Indeed $\log(\lambda(f))$
provides an upper bound for the topological entropy of
$f$ and is equal to it under natural assumptions 
(\emph{see} \cite{BedfordDiller, DinhSibony}). The dynamical
degree is invariant under conjugacy; as shown in 
\cite{BlancCantat} precise knowledge on~$\lambda(f)$
provides useful information on the conjugacy class
of $f$. By definition a \textsl{Pisot number} 
is an algebraic integer $\lambda\in]1,+\infty[$ whose
other Galois conjugates lie in the open unit disk; 
Pisot numbers include integers $d\geq 2$ as well as
reciprocal quadratic integers $\lambda>1$. A 
\textsl{Salem number} is an algebraic integer 
$\lambda\in]1,+\infty[$ whose other Galois conjugates
are in the closed unit disk, with at least one on 
the boundary. Diller and Favre proved the following
statement:

\begin{thm}[\cite{DillerFavre}]
Let $f$ be a birational 
self-map of a complex projective surface. 

If $\lambda(f)$ is different from $1$, then 
$\lambda(f)$ is a Pisot number or a Salem number.
\end{thm}

One of the goal of \cite{BlancCantat} is the 
study of the structure of the set of all 
dynamical degrees $\lambda(f)$ where~$f$
runs over the group of birational maps 
$\mathrm{Bir}(S)$ and $S$ over the collection
of all projective surfaces. In particular 
they get:

\begin{thm}[\cite{BlancCantat}]
Let $\Lambda$ be the set of all dynamical degrees
of birational maps of complex projective surfaces. 
Then 
\begin{itemize}
\item[$\diamond$] $\Lambda$ is a well ordered 
subset of $\mathbb{R}_+;$

\item[$\diamond$] if $\lambda$ is an element of 
$\Lambda$, there is a real number 
$\varepsilon>0$ such that 
$]\lambda,\lambda+\varepsilon]$ does not 
intersect $\Lambda;$

\item[$\diamond$] there is a non-empty interval 
$]\lambda_G,\lambda_G+\varepsilon]$, with 
$\varepsilon>0$, on the right of the golden mean
that contains infinitely many Pisot and Salem
numbers, but does not contain any dynamical 
degree.
\end{itemize}
\end{thm}

\medskip

\item[$\diamond$] \textbf{Dynamical number of base-points} 
(\cite{BlancDeserti}).

If $S$ is a projective smooth surface, every 
$f\in\mathrm{Bir}(S)$ admits a resolution
\[
\xymatrix{
&Z \ar[dl]_{\pi_1} \ar[dr]^{\pi_2} &\\
S\ar@{-->}[rr]_f &  & S }
\]
where $\pi_1$, $\pi_2$ are sequences of blow-ups. 
The resolution is \textsl{minimal} if and only if
no $(-1)$-curve of~$Z$ are contracted by both 
$\pi_1$ and $\pi_2$. Assume that the 
resolution is minimal; the \textsl{base-points} of
$f$ are the points blown-up by $\pi_1$, which 
can be points of $S$ or infinitely near points.
We denote by $\mathfrak{b}(f)$ the number of 
such points, which is also equal to the difference
of the ranks of $\mathrm{Pic}(Z)$ and 
$\mathrm{Pic}(S)$, and thus equal to 
$\mathfrak{b}(f^{-1})$.

Let us define the \textsl{dynamical number of 
base-points} of $f$ by
\[
\mu(f)=\displaystyle\lim_{k\to +\infty}\frac{\mathfrak{b}(f^k)}{k}.
\]
Since 
$\mathfrak{b}(f\circ\varphi)\leq\mathfrak{b}(f)+\mathfrak{b}(\varphi)$
for any $f$, $\varphi\in\mathrm{Bir}(S)$ we see
that $\mu(f)$ is a non-negative real number. 
Moreover $\mathfrak{b}(f^{-1})$ and 
$\mathfrak{b}(f)$ being equal we get
$\mu(f^k)=\vert k\mu(f)\vert$ for any 
$k\in\mathbb{Z}$. Furthermore the dyna\-mical 
number of base-points is an invariant of conjugation:
if $\psi\colon S\dashrightarrow Z$ is a birational
map between smooth projective surfaces and 
if~$f$ belongs to $\mathrm{Bir}(S)$, then 
$\mu(f)=\mu(\psi\circ f\circ\psi^{-1})$. In 
particular if $f$ is conjugate to an automorphism
of a smooth projective surface, then $\mu(f)=0$.
The converse holds, {\it i.e.} $f\in\mathrm{Bir}(S)$
is conjugate to an automorphism of a smooth 
projective surface if and only if $\mu(f)=0$
(\cite[Proposition 3.5]{BlancDeserti}). 
This follows from the geometric interpretation
of $\mu$ we will recall now. If $f\in\mathrm{Bir}(S)$
is a birational map, a (possibly infinitely near)
base-point~$p$ of $f$ is a 
\textsl{persistent base-point} of $f$ if there
exists an integer $N$ such that $p$ is a base-point
of $f^k$ for any $k\geq N$ but is not a base-point
of $f^{-k}$ for any $k\geq N$. We put an 
equivalence relation on the set of points that 
belongs to $S$ or are infinitely near: take a 
minimal resolution of $f$ 
\[
\xymatrix{
&Z \ar[dl]_{\pi_1} \ar[dr]^{\pi_2} &\\
S\ar@{-->}[rr]_f &  & S }
\]
where $\pi_1$, $\pi_2$ are sequences of blow-ups; the 
point $p$ is \textsl{equivalent} to $q$ if there 
exists an integer $k$ such that 
$(\pi_2\circ\pi_1^{-1})^k(p)=q$.
Denote by $\nu$ the number of equivalence classes of
persistent base-points of $f$; then the set 
\[
\big\{\mathfrak{b}(f^k)-\nu k\,\vert\,k\geq 0\big\}\subset\mathbb{Z}
\]
is bounded. In particular $\mu(f)$ is an integer, equal 
to $\nu$ (\emph{see} \cite[Proposition 3.4]{BlancDeserti}). 
This gives a bound for $\mu(f)$; indeed if
$f\in\mathrm{Bir}(\mathbb{P}^2_\mathbb{C})$ is a map 
whose base-points have multiplicities $m_1\geq m_2\geq \ldots\geq m_r$
then (\emph{see for instance} \cite[\S 2.5]{AlberichCarraminana} and \cite[Corollary 2.6.7]{AlberichCarraminana})
\[
\left\{
\begin{array}{lll}
\displaystyle\sum_{i=1}^rm_i=3(\deg(f)-1) \\
\displaystyle\sum_{i=1}^rm_i^2=\deg(f)^2-1 \\
m_1+m_2+m_3\geq \deg(f)+1 
\end{array}
\right.
\]
in particular $r\leq 2\deg(f)-1$ so
$\nu\leq 2\deg(f)-1$ and $\mu(f)\leq 2\deg(f)-1$. 

If $f\in\mathrm{Bir}(\mathbb{P}^2_\mathbb{C})$
is a Jonqui\`eres twist, then there exists an integer
$a\in\mathbb{N}$ such that 
\[
\displaystyle\lim_{k\to +\infty}\frac{\deg(f^k)}{k}=a^2\,\frac{\mu(f)}{2};
\]
moreover $a$ is the degree of the curves of the 
unique pencil of rational curves inva\-riant by $f$
(\emph{see} \cite[Proposition 4.5]{BlancDeserti}).
In particular $a=1$ if and only if $f$ preserves a 
pencil of lines. On the one hand 
$\big\{\mu(f)\,\vert\,f\in\mathrm{Bir}(\mathbb{P}^2_\mathbb{C})\big\}\subseteq\mathbb{N}$ 
and on the other hand if $f$ belongs to $\mathcal{J}$,
then $\mu(f)>0$; as a result 
\[
\big\{\mu(f)\,\vert\,f\in\mathcal{J}\big\}\subseteq\mathbb{N}\smallsetminus\{0\}.
\]
Let us recall that if 
$f_{\alpha,\beta}=\left(\frac{\alpha x+y}{x+1},\beta y\right)$
then $\mu(f_{\alpha,\beta})=1$. Indeed, by induction one can prove
that 
$f_{\alpha,\beta}^{2n}=\left(\frac{P_n(x,y)}{Q_n(x,y)},\beta^{2n}y\right)$ with 
\begin{align*}
& P_n(x,y)=\displaystyle\sum_{0\leq i+j\leq n+1}a_{ij}x^iy^j && 
Q_n(x,y)=\displaystyle\sum_{0\leq i+j\leq n}b_{ij}x^iy^j
\end{align*} 
and $a_{ij}\geq 0$, $b_{ij}\geq 0$ for any $n\geq 0$, so that 
$\deg f_{\alpha,\beta}^{2n}=n+1$ for any $n\geq 0$; we conclude using the 
fact that $\mu(f)=2\displaystyle\lim_{k\to +\infty}\frac{\deg(f^k)}{k}$.
Furthermore 
$\mu(f_{\alpha,\beta}^k)=\vert k\mu(f_{\alpha,\beta})\vert=\vert k\vert$
for any $k\in\mathbb{Z}$. Hence 
\[
\big\{\mu(f)\,\vert\,f\in\mathcal{J}\big\}=\mathbb{N}\smallsetminus\{0\}
\]
and
\[
\big\{\mu(f)\,\vert\,f\in\mathrm{Bir}(\mathbb{P}^2_\mathbb{C})\big\}=\mathbb{N}.
\]
As we have seen if $f$ belongs to $\mathcal{J}$, then 
$\mu(f)=2\displaystyle\lim_{k\to +\infty}\frac{\deg(f^k)}{k}$. 
Can we express $\mu(f)$ in a simpliest way ? We will see that 
if $f$ is a non base-wandering Jonqui\`eres twist, the 
answer is yes.

\medskip

\item[$\diamond$] \textbf{Results}.

The dynamical number of base-points of birational self maps of the
complex projective plane satisfies the following properties:

\begin{theoalph}\label{thm:faitsdivers}
\begin{itemize}
\item[$\diamond$] If $f$ is a birational self-map from 
$\mathbb{P}^2_\mathbb{C}$ into itself, then its dynamical number
of base-points is bounded: if $f\in\mathrm{Bir}(\mathbb{P}^2_\mathbb{C})$, 
then $\mu(f)\leq 2\deg(f)-1$.

\item[$\diamond$] We can  
precise the set of all dynamical 
numbers of base-points of birational maps of 
$\mathbb{P}^2_\mathbb{C}$ $($resp. of 
Jonqui\`eres maps of $\mathbb{P}^2_\mathbb{C}):$ 
\begin{align*}
&\big\{\mu(f)\,\vert\,f\in\mathrm{Bir}(\mathbb{P}^2_\mathbb{C})\big\}=\mathbb{N}
&& \text{and}
&&
\big\{\mu(f)\,\vert\,f\in\mathcal{J}\big\}=\mathbb{N}\smallsetminus\{0\}.
\end{align*}

\item[$\diamond$] There exist sequences $(f_n)_n$ of birational
self-maps of $\mathbb{P}^2_\mathbb{C}$ such that
\begin{itemize}
\item[$\bullet$] $\mu(f_n)>0$ for any $n\in\mathbb{N}$; 

\item[$\bullet$] $\mu\big(\displaystyle\lim_{n\to +\infty}f_n\big)=0$.
\end{itemize}

\item[$\diamond$] There exist sequences $(f_n)_n$ of birational self-maps of $\mathbb{P}^2_\mathbb{C}$ such that
\begin{itemize}
\item[$\bullet$] $\mu(f_n)=0$ for any $n\in\mathbb{N}$; 

\item[$\bullet$] $\mu\big(\displaystyle\lim_{n\to +\infty}f_n\big)>0$.
\end{itemize}
\end{itemize}
\end{theoalph}

\smallskip

Let us now give a formula to determine the 
dynamical number of base-points of Jonqui\`eres
twists that preserves fiberwise the fibration. 

\begin{theoalph}\label{thm:main}
Let $f=\left(\frac{A(y)x+B(y)}{C(y)x+D(y)},y\right)$ 
be a Jonqui\`eres twist that preserves fiberwise the
fibration, and let $M_f$ be its associated matrix. Denote
by $\mathrm{Tr}(M_f)$ the trace $M_f$,
by $\chi_f$ the characteristic polynomial of~$M_f$, 
and by $\Delta_f$ the discriminant of $\chi_f$. Then 
exactly one of the following holds:
\smallskip
\begin{itemize}
\item[1.] If $\chi_f$ has two distinct roots in 
$\mathbb{C}[y]$,
then  $f$ is conjugate to 
$g=\left(\frac{\mathrm{Tr}(M_f)+\delta_f}{\mathrm{Tr}(M_f)-\delta_f}\,x,y\right)$, where $\delta_f^2=\Delta_f$, and
\[
\mu(f)=\mu(g)=2(\deg (g)-1).
\]

\item[2.] If $\chi_f$ has no root in $\mathbb{C}[y]$,
set
\[
\Omega_f=\mathrm{gcd}\left(\frac{\mathrm{Tr}(M_f)}{2},\left(\frac{\mathrm{Tr}(M_f)}{2}\right)^2-\det(M_f)\right)
\]
and let us define $P_f$ and $S_f$ as
\begin{align*}
& \frac{\mathrm{Tr}(M_f)}{2}=P_f\Omega_f, && \left(\frac{\mathrm{Tr}(M_f)}{2}\right)^2-\det(M_f)=S_f\Omega_f.
\end{align*}

\smallskip

\begin{enumerate}
\item[2.a.]  If $\mathrm{gcd}(\Omega_f,S_f)=1$, 
then
\begin{enumerate}
\item[$\diamond$]  if $\deg(S_f)\leq \deg(\Omega_f)+2\deg(P_f)$, then
$\mu(f)=\deg(\Omega_f)+2\deg(P_f)$;
\item[$\diamond$]  otherwise $\mu(f)=\deg(S_f)$.
\end{enumerate}

\smallskip

\item[2.b.] If $S_f=\Omega_f^pT_f$ with $p\geq 1$ and 
$\mathrm{gcd}(T_f,\Omega_f)=1$, then
\begin{enumerate}
\item[$\diamond$]  if $\deg(S_f)\leq \deg(\Omega_f)+2\deg(P_f)$, then
$\mu(f)=2\deg(P_f)$;
\item[$\diamond$] otherwise
$\mu(f)=\deg(S_f)-\deg(\Omega_f)$.
\end{enumerate}

\smallskip

\item[2.c.]  If $\Omega_f=S_f^p T_f$ with $p\geq 1$ and 
$\mathrm{gcd}(T_f,S_f)=1$, then
$\mu(f)=2\deg(P_f)+\deg(\Omega_f)-\deg(S_f).$
\end{enumerate}
\end{itemize}
\end{theoalph}

As a consequence we are able to determine the 
dynamical number of base-points of 
non base-wandering Jonqui\`eres twists:

\begin{coralph}
Let $f=(f_1,f_2)$ be a non base-wandering Jonqui\`eres twist.

If $\ell$ is the order of $f_2$, then
$\mu(f)=~\frac{\mu(f^\ell)}{\ell}$ where $\mu(f^\ell)$
is given by Theorem \ref{thm:main}. 
\end{coralph}

Combining the inequalities obtained in 
Theorem \ref{thm:faitsdivers} and Theorem \ref{thm:main} 
we get the following statement (we use the notations
introduced in Theorem \ref{thm:main}):

\begin{coralph}
Let $f$ be a Jonqui\`eres twist that preserves
fiberwise the fibration. 
Assume that $\chi_f$ has two distinct roots in 
$\mathbb{C}[y]$. 

Then there exists a conjugate $g$ of $f$ such that $g$
belongs to $\mathrm{J}_m$ and $\deg(g)\leq\deg(f)$.
For instance 
$g=\left(\frac{\mathrm{Tr}(M_f)+\delta_f}{\mathrm{Tr}(M_f)-\delta_f}\,x,y\right)$ suits.
\end{coralph}
\end{itemize}

\section{Dynamical number of base-points of Jonqui\`eres twists}

In this section we will prove Theorem \ref{thm:main}.

Let $f$ be an element of $\mathrm{J}_0$; write $f$ 
as $\left(\frac{A(y)x+B(y)}{C(y)x+D(y)},y\right)$ 
with $A$, $B$, $C$, $D\in\mathbb{C}[y]$. The 
characteristic polynomial of 
$M_f=\left(
\begin{array}{cc}
A & B\\
C & D
\end{array}
\right)$ is $\chi_f(X)=X^2-\mathrm{Tr}(M_f)X+\det(M_f)$.
There are three possibilities:
\begin{enumerate}
\item $\chi_f$ has one root of multiplicity $2$ in $\mathbb{C}[y]$;

\item $\chi_f$ has two distinct roots in $\mathbb{C}[y]$;

\item $\chi_f$ has no root in $\mathbb{C}[y]$.
\end{enumerate}

Let us consider these three possibilities.
\begin{enumerate}
\item If $\chi_f$ has one root of multiplicity $2$
in $\mathbb{C}[y]$, then $f$ is conjugate to 
the elliptic birational map
$(x+a(y),y)$ of $\mathrm{J}_a$. In 
particular $f$ does not belong to $\mathcal{J}$. 

\bigskip

\item Assume that $\chi_f$ has two distinct roots. The
discriminant of $\chi_f$ is 
\[
\Delta_f=\big(\mathrm{Tr}(M_f)\big)^2-4\det(M_f)=\delta_f^2
\]
and the roots of $\chi_f$ are 
\begin{align*}
& \frac{\mathrm{Tr}(M_f)+\delta_f}{2} && \text{and} && \frac{\mathrm{Tr}(M_f)-\delta_f}{2}.
\end{align*}
Furthermore $M_f$ is conjugate to 
$\left(
\begin{array}{cc}
\frac{\mathrm{Tr}(M_f)+\delta_f}{2} & 0 \\
0 & \frac{\mathrm{Tr}(M_f)-\delta_f}{2}
\end{array}\right)$, {\it i.e.} $f$ is conjugate 
to $g=(a(y)x,y)\in\mathrm{J}_m$ with 
$a(y)=\frac{\mathrm{Tr}(M_f)+\delta_f}{\mathrm{Tr}(M_f)-\delta_f}$.
Let us first express $\mu(g)$ thanks to $\deg(g)$. 
Remark that $g^k=(a(y)^kx,y)$. Write $a(y)^j$ as 
$\frac{P_j(y)}{Q_j(y)}$ where $P_j$, $Q_j\in\mathbb{C}[y]$,
$\mathrm{gcd}(P_j,Q_j)=1$,
then $\deg(g^j)=\max(\deg(P_j),\deg(Q_j))+1$. 
But $\deg(P_j)=j\deg(P)$ and $\deg(Q_j)=j\deg(Q)$
so 
\[
\deg(g^k)=\max(k\deg(P_f),k\deg(Q_1))+1=k\underbrace{\max(\deg(P_f),\deg(Q_1))}_{\deg(g)-1}+1.
\]
As a consequence $\deg(g^k)=k\deg(g)-k+1$. According to
$\mu(g)=2\displaystyle\lim_{k\to +\infty}\frac{\deg(g^k)}{k}$
we get 
\[
\mu(g)=2\displaystyle\lim_{k\to+\infty}\Big(\deg(g)-1+\frac{1}{k}\Big)=2(\deg(g)-1).
\]

Let us now express $\mu(f)$ thanks to $f$. Since $f$ and
$g$ are conjugate $\mu(f)=\mu(g)$ hence $\mu(f)=2(\deg(g)-1)$.
But $g=\left(\frac{\mathrm{Tr}(M_f)+\delta_f}{\mathrm{Tr}(M_f)-\delta_f}x,y\right)$; in particular
\[
\deg(g)\leq 1+\max\big(\deg(\mathrm{Tr}(M_f)+\delta_f),\deg(\mathrm{Tr}(M_f)-\delta_f)\big) 
\]
and $\mu(f)\leq 2\max\big(\deg(\mathrm{Tr}(M_f)+\delta_f),\deg(\mathrm{Tr}(M_f)-\delta_f)\big)$.

\bigskip

\item Suppose that $\chi_f$ has no root in $\mathbb{C}[y]$.
This means that $\big(\mathrm{Tr}(M_f)\big)^2-4\det(M_f)$
is not a square in $\mathbb{C}[y]$ (hence $BC\not=0$). 
Note that 
\[
\left(
\begin{array}{cc}
C & \frac{D-A}{2}\\
0 & 1
\end{array}
\right)
\left(
\begin{array}{cc}
A & B\\
C & D
\end{array}
\right)
\left(
\begin{array}{cc}
C & \frac{D-A}{2}\\
0 & 1
\end{array}
\right)^{-1}=\left(
\begin{array}{cc}
\frac{\mathrm{Tr}(M_f)}{2} & \left(\frac{\mathrm{Tr}(M_f)}{2}\right)^2-\det(M_f) \\
1 & \frac{\mathrm{Tr}(M_f)}{2}
\end{array}
\right).
\]
In other words $f$ is conjugate to 
\[
g=\left(\frac{\frac{\mathrm{Tr}(M_f)}{2}x+\left(\frac{\mathrm{Tr}(M_f)}{2}\right)^2-\det(M_f)}{x+\frac{\mathrm{Tr}(M_f)}{2}},y\right)\in\mathrm{J}_{\frac{\mathrm{Tr}(M_f)}{2}}.
\]
Set $P(y)=\frac{\mathrm{Tr}(M_f)}{2}\in\mathbb{C}[y]$ and 
$F(y)=\left(\frac{\mathrm{Tr}(M_f)}{2}\right)^2-\det(M_f)\in\mathbb{C}[y]$,
{\it i.e.} $f$ is conjugate to 
$g=\left(\frac{P(y)x+F(y)}{x+P(y)},y\right)$ with 
$P$, $F\in\mathbb{C}[y]$. 
Denote by $d_P$ (resp. $d_F$) the degree of $P$ 
(resp. $F$).
Remark that 
$\deg(g)=\max(d_P+1,d_F,2)$. 

Let us now express $\deg(g^k)$.
Consider $M_g=\left(
\begin{array}{cc}
P & F\\
1 & P
\end{array}
\right)$ and set 
\begin{align*}
& Q=\left(
\begin{array}{cc}
\sqrt{F} & -\sqrt{F}\\
1 & 1
\end{array}
\right)&& \text{and} && D=\left(
\begin{array}{cc}
P+\sqrt{F} & 0\\
0 & P-\sqrt{F}
\end{array}
\right)
\end{align*}
Then $M_g^k=QD^kQ^{-1}$ hence 
\[
M_g^k=\left(
\begin{array}{cc}
\sqrt{F}\frac{(P+\sqrt{F})^k+(P-\sqrt{F})^k}{(P+\sqrt{F})^k-(P-\sqrt{F})^k} & F\\
1 &  \sqrt{F}\frac{(P+\sqrt{F})^k+(P-\sqrt{F})^k}{(P+\sqrt{F})^k-(P-\sqrt{F})^k}
\end{array}
\right)
\]

\smallskip

Let us set 
\[
\Upsilon_k=\sqrt{F}\frac{(P+\sqrt{F})^k+(P-\sqrt{F})^k}{(P+\sqrt{F})^k-(P-\sqrt{F})^k}
\]
and let us denote by $D_k$ (resp. $N_k$) the denominator
(resp. numerator) of $\Upsilon_k$.

\begin{lem}\label{lem:tec1}
Let $\Omega_f=\mathrm{gcd}(P,F)$ and write $P$ $($resp. $F)$ as $\Omega_f P_f$ 
$($resp. $\Omega_f S_f)$. Assume $\mathrm{gcd}(S_f,\Omega_f)=1$.
Then 
\begin{itemize}
\item[$\diamond$] if $d_{S_f}\leq d_{\Omega_f}+2d_{P_f}$,
then $\mu(g)=d_{\Omega_f}+2d_{P_f}$;

\item[$\diamond$] otherwise $\mu(g)=d_{S_f}$.
\end{itemize}
\end{lem}

\begin{proof}
\begin{enumerate}
\item Assume $k$ even, write $k$ as $2\ell$.
A straightforward computation yields to 
\[
\Upsilon_{2\ell}=\frac{\displaystyle\sum_{j=0}^{\ell}\binom{2\ell}{2j}\Omega_f^{\ell-j}P_f^{2(\ell-j)}S_f^j}{P_f\displaystyle\sum_{j=0}^{\ell-1}\binom{2\ell}{2j+1}\Omega_f^{\ell-1-j}P_f^{2(\ell-1-j)}S_f^j}
\]

Recall that
$\mathrm{gcd}(\Omega_f,S_f)=1$ by assumption and $\mathrm{gcd}(\Omega_f,P_f)=1$
by construction.
On the one hand 
\[
\deg(N_{2\ell})=\left\{
\begin{array}{ll}
\ell(d_{\Omega_f}+2d_{P_f})\quad\text{if $d_{S_f}\leq d_{\Omega_f}+2d_{P_f}$}\\
\ell d_{S_f}\quad\text{otherwise}\\
\end{array}
\right.
\]
On the other hand
\[
\deg(D_{2\ell})=\left\{
\begin{array}{ll}
d_{P_f}+\left(\ell-1\right)(d_{\Omega_f}+2d_{P_f})\quad\text{if $d_{S_f}\leq d_{\Omega_f}+2d_{P_f}$}\\
d_{P_f}+\left(\ell-1\right)d_{S_f}\quad\text{otherwise}\\
\end{array}
\right.
\]
Finally
\begin{small}
\[
\deg(g^{2\ell})=\left\{
\begin{array}{ll}
\max\Big(\ell(d_{\Omega_f}+2d_{P_f})+1,d_{S_f}+\ell d_{\Omega_f}+({2\ell}-1)d_{P_f},\left(\ell-1\right)d_{\Omega_f}+({2\ell}-1)d_{P_f}+2\Big)\quad\text{if $d_{S_f}\leq d_{\Omega_f}+2d_{P_f}$}\\
\max\Big(\ell d_{S_f}+1,d_{\Omega_f}+d_{P_f}+\ell d_{S_f},d_{P_f}+\left(\ell-1\right)d_{S_f}+2\Big)\quad\text{otherwise}\\
\end{array}
\right.
\]
\end{small}

\item Suppose $k$ odd, write $k$ as $2\ell+1$.
A straightforward computation yields to 
\[
\Upsilon_{2\ell+1}=P_f\Omega_f\,\,\frac{\displaystyle\sum_{j=0}^{\ell}\binom{2\ell+1}{2j}\Omega_f^{\ell-j}P_f^{2(\ell-j)}S_f^j}{\displaystyle\sum_{j=0}^{\ell}\binom{2\ell+1}{2j+1}\Omega_f^{\ell-j}P_f^{2(\ell-j)}S_f^j}
\]

Let us recall that
$\mathrm{gcd}(\Omega_f,S_f)=1$ by assumption and $\mathrm{gcd}(\Omega_f,P_f)=1$
by construction.

On the one hand
\[
\deg(N_{2\ell+1})=\left\{
\begin{array}{ll}
\ell(d_{\Omega_f}+2d_{P_f})+d_{\Omega_f}+d_{P_f}\quad\text{if $d_{S_f}\leq d_{\Omega_f}+2d_{P_f}$}\\
\ell d_{S_f}+d_{P_f}+d_{\Omega_f}\quad\text{otherwise}\\
\end{array}
\right.
\]
On the other hand
\[
\deg(D_{2\ell+1})=\left\{
\begin{array}{ll}
\ell (d_{\Omega_f}+2d_{P_f})\quad\text{if $d_{S_f}\leq d_{\Omega_f}+2d_{P_f}$}\\
\ell d_{S_f}\quad\text{otherwise}\\
\end{array}
\right.
\]
Finally 
\begin{small}
\[
\deg(g^{2\ell+1})=\left\{
\begin{array}{ll}
\max\Big((\ell+1) d_{\Omega_f}+(2\ell+1) d_{P_f}+1,(\ell+1)d_{\Omega_f}+2\ell d_{P_f}+d_{S_f},\ell (d_{\Omega_f}+2d_{P_f})+2\Big)\quad\text{if $d_{S_f}\leq d_{\Omega_f}+2d_{P_f}$}\\
\max\Big(\ell d_{S_f}+d_{P_f}+d_{\Omega_f}+1,(\ell+1) d_{S_f}+d_{\Omega_f},\ell d_{S_f}+2\Big)\quad\text{otherwise}\\
\end{array}
\right.
\]
\end{small}
\end{enumerate}

We conclude with the equality $\mu(g)=2\displaystyle\lim_{k\to +\infty}\frac{\deg(g^k)}{k}$.
\end{proof}

\begin{lem}\label{lem:tec2}
Let $\Omega_f=\mathrm{gcd}(P,F)$ and write $P$ $($resp. $F)$ as $\Omega_f P_f$ 
$($resp. $\Omega_f S_f)$. Suppose that 
$S_f=\Omega_f^p T_f$ with $p\geq 1$ and $\mathrm{gcd}(T_f,\Omega_f)=1$.
Then 
\begin{itemize}
\item[$\diamond$] if $d_{S_f}\leq d_{\Omega_f}+2d_{P_f}$,
then $\mu(g)=2d_{P_f}$;

\item[$\diamond$] otherwise $\mu(g)=d_{S_f}-d_{\Omega_f}$.
\end{itemize}

\end{lem}

\begin{proof}
\begin{enumerate}
\item Assume $k$ even, write $k$ as $2\ell$. We get
\begin{small}
\[
\Upsilon_{2\ell}=\frac{\displaystyle\sum_{j=0}^{\ell}\binom{2\ell}{2j}\Omega_f^{\ell-j}P_f^{2(\ell-j)}S_f^j}{P_f\displaystyle\sum_{j=0}^{\ell-1}\binom{2\ell}{2j+1}\Omega_f^{\ell-1-j}P_f^{2(\ell-1-j)}S_f^j}=\frac{\Omega_f\displaystyle\sum_{j=0}^{\ell}\binom{2\ell}{2j}\Omega_f^{(p-1)j}P_f^{2(\ell-j)}T_f^j}{P_f\displaystyle\sum_{j=0}^{\ell-1}\binom{2\ell}{2j+1}\Omega_f^{(p-1)j}P_f^{2(\ell-1-j)}T_f^j}
\]
\end{small}

Recall that
$\mathrm{gcd}(\Omega_f,T_f)=1$ and that $d_{S_f}=pd_{\Omega_f}+d_{T_f}$, {\it 
i.e.} $d_{T_f}=d_{S_f}-pd_{\Omega_f}$.
On the one hand 
\[
\deg(N_{2\ell})=\left\{
\begin{array}{ll}
2\ell d_{P_f}+d_{\Omega_f}\quad\text{if $d_{S_f}\leq d_{\Omega_f}+2d_{P_f}$}\\
\ell d_{S_f}+(1-\ell)d_{\Omega_f}\quad\text{otherwise}\\
\end{array}
\right.
\]
On the other hand
\[
\deg(D_{2\ell})=\left\{
\begin{array}{ll}
(2\ell-1)d_{P_f}\quad\text{if $d_{S_f}\leq d_{\Omega_f}+2d_{P_f}$}\\
(\ell-1)(d_{S_f}-d_{\Omega_f})+d_{P_f}\quad\text{otherwise}\\
\end{array}
\right.
\]
Finally 
\begin{small}
\[
\deg(g^{2\ell})=\left\{
\begin{array}{ll}
\max\Big(2\ell d_{P_f}+d_{\Omega_f}+1,(2\ell-1)d_{P_f}+d_{\Omega_f}+d_{S_f},(2\ell-1)d_{P_f}+1\Big)\quad\text{if $d_{S_f}\leq d_{\Omega_f}+2d_{P_f}$}\\
\max\Big(\ell d_{S_f}-(\ell-1)d_{\Omega_f}+1,\ell d_{S_f}+(2-\ell)d_{\Omega_f}+d_{P_f},(\ell-1)(d_{S_f}-d_{\Omega_f})+d_{P_f}+1\Big)\quad\text{otherwise}\\
\end{array}
\right.
\]
\end{small}

\item Suppose $k$ odd, write $k$ as $2\ell+1$.
We get
\begin{small}
\[
\Upsilon_{2\ell+1}=\frac{P_f\Omega_f\displaystyle\sum_{j=0}^{\ell}\binom{2\ell+1}{2j}\Omega_f^{\ell-j}P_f^{2(\ell-j)}S_f^j}{\displaystyle\sum_{j=0}^{\ell}\binom{2\ell+1}{2j+1}\Omega_f^{\ell-j}P_f^{2(\ell-j)}S_f^j}=\frac{P_f\Omega_f\displaystyle\sum_{j=0}^{\ell}\binom{2\ell+1}{2j}\Omega_f^{(p-1)j}P_f^{2(\ell-j)}T_f^j}{\displaystyle\sum_{j=0}^{\ell}\binom{2\ell+1}{2j+1}\Omega_f^{(p-1)j}P_f^{2(\ell-j)}T_f^j}
\]
\end{small}

On the one hand
\[
\deg(N_{2\ell+1})=\left\{
\begin{array}{ll}
(2\ell+1)d_{P_f}+d_{\Omega_f}\quad\text{if $d_{S_f}\leq d_{\Omega_f}+2d_{P_f}$}\\
\ell d_{S_f}-(\ell-1)d_{\Omega_f}+d_{P_f}\quad\text{otherwise}\\
\end{array}
\right.
\]
On the other hand
\[
\deg(D_{2\ell+1})=\left\{
\begin{array}{ll}
2\ell d_{P_f}\quad\text{if $d_{S_f}\leq d_{\Omega_f}+2d_{P_f}$}\\
\ell d_{S_f}-\ell d_{\Omega_f}\quad\text{otherwise}\\
\end{array}
\right.
\]
Finally 
\begin{small}
\[
\deg(g^{2\ell+1})=\left\{
\begin{array}{ll}
\max\Big((2\ell+1)d_{P_f}+d_{\Omega_f}+1,2\ell d_{P_f}+d_{\Omega_f}+d_{S_f},2\ell d_{P_f}+1\Big)\quad\text{if $d_{S_f}\leq d_{\Omega_f}+2d_{P_f}$}\\
\max\Big(\ell d_{S_f}-(\ell-1)d_{\Omega_f}+d_{P_f}+1,(\ell+1) d_{S_f}-(\ell-1) d_{\Omega_f},\ell d_{S_f}-\ell d_{\Omega_f}+1\Big)\quad\text{otherwise}\\
\end{array}
\right.
\]
\end{small}
\end{enumerate}

We conclude with the equality $\mu(g)=2\displaystyle\lim_{k\to +\infty}\frac{\deg(g^k)}{k}$.
\end{proof}

\begin{lem}\label{lem:tec2}
Let $\Omega_f=\mathrm{gcd}(P,F)$ and write $P$ $($resp. $F)$ as $\Omega_f P_f$ 
$($resp. $\Omega_f S_f)$. Suppose that 
$\Omega_f=S_f^p T_f$ with $p\geq 1$ and $\mathrm{gcd}(T_f,S_f)=1$.
Then  
\[
\mu(g)=2d_{P_f}+d_{\Omega_f}-d_{S_f}.
\]
\end{lem}

\begin{proof}
\begin{enumerate}
\item Assume $k$ even, write $k$ as $2\ell$. We obtain
\begin{small}
\[
\Upsilon_{2\ell}=\frac{\displaystyle\sum_{j=0}^{\ell}\binom{2\ell}{2j}\Omega_f^{\ell-j}P_f^{2(\ell-j)}S_f^j}{P_f\displaystyle\sum_{j=0}^{\ell-1}\binom{2\ell}{2j+1}\Omega_f^{\ell-1-j}P_f^{2(\ell-1-j)}S_f^j}=\frac{S_f\displaystyle\sum_{j=0}^{\ell}\binom{2\ell}{2j}S_f^{j(p-1)}P_f^{2j}T_f^j}{P_f\displaystyle\sum_{j=0}^{\ell-1}\binom{2\ell}{2j+1}S_f^{j(p-1)}P_f^{2j}T_f^j}
\]
\end{small}

Recall that
$\mathrm{gcd}(S_f,T_f)=1$; one has
\[
\deg(N_{2\ell})=(p\ell-\ell+1)d_{S_f}+2\ell d_{P_f}+\ell d_{T_f}
\]
and
\[
\deg(D_{2\ell})=(\ell-1)(p-1)d_{S_f}+(2\ell-1)d_{P_f}+(\ell-1)d_{T_f}
\]
Finally 
\begin{align*}
&\deg(g^{2\ell})=
\max\Big((p\ell-\ell+1)d_{S_f}+2\ell d_{P_f}+\ell d_{T_f}+1,\\
& \hspace{4cm}(\ell(p-1)+2)d_{S_f}+(2\ell-1)d_{P_f}+\ell d_{T_f},\\
& \hspace{4cm}(\ell-1)(p-1)d_{S_f}+(2\ell-1)d_{P_f}+(\ell-1)d_{T_f}+2\Big)
\end{align*}

\item Suppose $k$ odd, write $k$ as $2\ell+1$.
We get
\begin{small}
\[
\Upsilon_{2\ell+1}=\frac{\displaystyle\sum_{j=0}^{\ell}\binom{2\ell+1}{2j}P^{2\ell+1-2j}F^j}{\displaystyle\sum_{j=0}^{\ell}\binom{2\ell+1}{2j+1}P^{2\ell-2j}F^j}= \frac{S_f^pT_fP_f\displaystyle\sum_{j=0}^{\ell}\binom{2\ell+1}{2(\ell-j)}S_f^{j(p-1)} T_f^jP_f^{2j}}{\displaystyle\sum_{j=0}^{\ell}\binom{2\ell+1}{2j}S_f^{j(p-1)} T_f^jP_f^{2j}}.
\]
\end{small}

On the one hand
\[
\deg(N_{2\ell+1})=(p+\ell(p-1))d_{S_f}+(\ell+1) d_{T_f}+(2\ell+1) d_{P_f},
\]
and on the other hand
\[
\deg(D_{2\ell+1})=2\ell d_{P_f}+\ell d_{T_f}+\ell(p-1)d_{S_f}.
\]
Finally 
\begin{align*}
&  \deg(g^{2\ell+1})=\max\Big((p+\ell(p-1))d_{S_f}+(\ell+1) d_{T_f}+(2\ell+1) d_{P_f}+1,\\
& \hspace{4cm}
(p+1+\ell(p-1))d_{S_f}+2\ell d_{P_f}+(\ell+1)d_{T_f},\\
& \hspace{4cm}
2\ell d_{P_f}+\ell d_{T_f}+\ell(p-1)d_{S_f}+2\Big)
\end{align*}
\end{enumerate}

We conclude with the equality $\mu(g)=2\displaystyle\lim_{k\to +\infty}\frac{\deg(g^k)}{k}$.
\end{proof}
\end{enumerate}

\section{Examples}

In this section we will give examples that illustrate Theorem \ref{thm:main}; 
more precisely \S \ref {subsec:firstexample} 
(resp. \S \ref{subsec:secondexample}) 
illustrates Theorem \ref{thm:main}.1. (resp. Theorem \ref{thm:main}.2.)

\subsection{Example that illustrates Theorem \ref{thm:main}.1.}\label{subsec:firstexample}

\subsubsection{First example}

Consider the birational map of $\mathrm{J}$ given 
in the affine chart $x=1$ by $f=\big(y,(1-y)yz\big)$. 
The matrix associated to $f$ is
\[
M_f=\left(
\begin{array}{cc}
(1-y)y & 0\\
0 & 1
\end{array}
\right),
\]
and the Baum Bott index $\mathrm{BB}(f)$ of 
$f$ is $\frac{\big((1-y)y+1\big)^2}{(1-y)y}$; 
in particular $f$ belongs to $\mathcal{J}$ 
(Proposition \ref{pro:bb}). The characteristic
polynomial of $M_f$ is 
\[
\chi_f(X)=\big(X-(1-y)y\big)(X-1).
\]
According to Theorem \ref{thm:main}.1. 
one has $\mu(f)=4\leq 2\max\big(\deg(2),\deg(2(1-y)y\big)=4$.

\smallskip

We can see it another way: \cite{CerveauDeserti} 
asserts that $\deg(f^k)=k\deg(f)-k+1=3k-k+1=2k+1$. 
Consequently 
$\mu(f)=2\displaystyle\lim_{k\to +\infty}\frac{\deg(f^k)}{k}=2\displaystyle\lim_{k\to +\infty}\frac{2k+1}{k}=4$.

\smallskip

A third way to see this is to look at the configuration
of the exceptional divisors. For any $k\geq 1$ one has $f^k=\big(x^{2k+1}:x^{2k}y:(x-y)^ky^kz\big)$. The configuration
of the exceptional divisors of $f^k$ is 

\bigskip
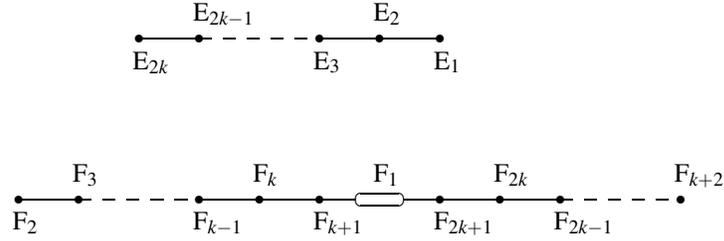
\begin{center}
\begin{figure}[H]
\setlength{\unitlength}{0.8cm}
\begin{small}
\begin{picture}(8,0.8)
\put(0,0){\circle*{0.15}}
\put(-0.1,-0.5){$\mathrm{E}_{2k}$}
\put(0,0){\line(1,0){1}}
\put(1,0){\circle*{0.15}}
\put(0.9,0.3){$\mathrm{E}_{2k-1}$}
\put(1,0){\line(1,0){0.2}}
\put(1.4,0){\line(1,0){0.2}}
\put(1.8,0){\line(1,0){0.2}}
\put(2.2,0){\line(1,0){0.2}}
\put(2.6,0){\line(1,0){0.2}}
\put(3,0){\circle*{0.15}}
\put(2.9,-0.5){$\mathrm{E}_3$}
\put(3,0){\line(1,0){1}}
\put(4,0){\circle*{0.15}}
\put(3.9,0.3){$\mathrm{E}_2$}
\put(4,0){\line(1,0){1}}
\put(5,0){\circle*{0.15}}
\put(4.9,-0.5){$\mathrm{E}_1$}
\end{picture}
\end{small}
\end{figure}
\end{center}

\begin{center}
\begin{figure}[H]
\setlength{\unitlength}{0.8cm}
\begin{small}
\begin{picture}(12,0.8)
\put(0,0){\circle*{0.15}}
\put(-0.1,-0.5){$\mathrm{F}_2$}
\put(0,0){\line(1,0){1}}
\put(1,0){\circle*{0.15}}
\put(0.9,0.3){$\mathrm{F}_3$}
\put(1,0){\line(1,0){0.2}}
\put(1.4,0){\line(1,0){0.2}}
\put(1.8,0){\line(1,0){0.2}}
\put(2.2,0){\line(1,0){0.2}}
\put(2.6,0){\line(1,0){0.2}}
\put(3,0){\circle*{0.15}}
\put(2.9,-0.5){$\mathrm{F}_{k-1}$}
\put(3,0){\line(1,0){1}}
\put(4,0){\circle*{0.15}}
\put(3.9,0.3){$\mathrm{F}_k$}
\put(4,0){\line(1,0){1}}
\put(5,0){\circle*{0.15}}
\put(4.9,-0.5){$\mathrm{F}_{k+1}$}
\put(5,0){\line(1,0){0.6}}
\put(6,0){\oval(0.8,0.2)}
\put(5.9,0.3){$\mathrm{F}_1$}
\put(6.4,0){\line(1,0){0.6}}
\put(7,0){\circle*{0.15}}
\put(6.9,-0.5){$\mathrm{F}_{2k+1}$}
\put(7,0){\line(1,0){1}}
\put(8,0){\circle*{0.15}}
\put(7.9,0.3){$\mathrm{F}_{2k}$}
\put(8,0){\line(1,0){1}}
\put(9,0){\circle*{0.15}}
\put(8.9,-0.5){$\mathrm{F}_{2k-1}$}
\put(9,0){\line(1,0){0.2}}
\put(9.4,0){\line(1,0){0.2}}
\put(9.8,0){\line(1,0){0.2}}
\put(10.2,0){\line(1,0){0.2}}
\put(10.6,0){\line(1,0){0.2}}
\put(11,0){\circle*{0.15}}
\put(10.9,0.3){$\mathrm{F}_{k+2}$}
\end{picture}
\end{small}
\end{figure}
\end{center}

where
\begin{itemize}
\item[$\diamond$] two curves are related by an edge if their intersection
is positive;

\item[$\diamond$] the self-intersections correspond to the shape of the
vertices;

\item[$\diamond$] the point means self-intersection $-1$, the rectangle
means self-intersection $-2k$.
\end{itemize}

In particular the number of base-points of $f^k$ is $2k+2k+1=4k+1$ and
\[
\mu(f)=\displaystyle\lim_{k\to +\infty}\frac{\#\mathfrak{b}(f^k)}{k}=4.
\]

\subsubsection{Second example}

Consider the birational map of $\mathrm{J}$ given 
in the affine chart $z=1$ by $f=\big(x,xy+x(x-1)\big)$. 
The matrix associated to $f$ is
\[
M_f=\left(
\begin{array}{cc}
x & x(x-1)\\
0 & 1
\end{array}
\right);
\]
according to Proposition \ref{pro:bb} the map $f$ is a Jonqui\`eres 
twist (indeed 
$\mathrm{BB}(f)=\frac{(1+x)^2}{x}\in\mathbb{C}(x)\smallsetminus\mathbb{C}$). 
The characteristic polynomial of $M_f$ is $\chi_f(X) =\big(X-x\big)(X-1)$.
and $f$ is conjugate to $g=(x,xy)$.
According to Theorem \ref{thm:main}.1. 
one has 
\[
\mu(f)=\mu(g)=2(\deg(g)-1)=2\leq 2\max\big(\deg(2),\deg(2(1-y)y\big)=2.
\] 

\smallskip

We can see it another way: for any $k\geq 1$ one has
$f^k=\big(x,x^ky+x^{k+1}-x)\big)$ 
and thus $\deg(f^k)=k+1$. As a result 
$\mu(f)=2\displaystyle\lim_{k\to +\infty}\frac{\deg(f^k)}{k}=2\times 1=2$.

\newpage

\subsection{Examples that illustrate Theorem 
\ref{thm:main}.2.}\label{subsec:secondexample}

\subsubsection{First example}

Consider the map of $\mathrm{J}$ given 
in the affine chart $y=1$ by 
\[
f=\left(x,\frac{x(1-xz)}{z}\right).
\]
The matrix associated to $f$ is
\[
M_f=\left(
\begin{array}{cc}
-x^2 & x\\
 1 & 0 
\end{array}
\right),
\]
the Baum Bott index $\mathrm{BB}(f)$ of 
$f$ is $-x^3$ and $f$ belongs to $\mathcal{J}$
(Proposition \ref{pro:bb}). 

Theorem \ref{thm:main}.2.a. asserts that $\mu(f)=3$. We can see it another way: a computation gives $\deg(f^{2k})=3k+1$ and 
$\deg(f^{2k+1})=3(k+1)$ for any $k\geq 0$.
Since 
$\mu(f)=2\displaystyle\lim_{k\to +\infty}\frac{\deg(f^k)}{k}$
one gets $\mu(f)=3$.

\subsubsection{Second example}

Consider the map $f$ of $\mathrm{J}$ associated to 
the matrix
\[
M_f=\left(
\begin{array}{cc}
y & 2y^8\\
y & 1
\end{array}
\right).
\]
The Baum Bott index $\mathrm{BB}(f)$ of $f$ is 
$\frac{(y+1)^2}{y(1-2y^8)}$ and $f$ belongs to $\mathcal{J}$ 
(Proposition \ref{pro:bb}).
Theorem \ref{thm:main}.2.a. asserts that $\mu(f)=9$. 
We can see it another way: a computation gives $\deg(f^{2k})=9k+1$ and 
$\deg(f^{2k+1})=9k+8$ for any $k\geq 0$.
Since 
$2\displaystyle\lim_{k\to +\infty}\frac{\deg(f^k)}{k}=\mu(f)$
one gets $\mu(f)=9$.

\subsubsection{Third example}

Let us consider the Jonqui\`eres map 
of $\mathbb{P}^2_\mathbb{C}$ given in the affine 
chart $z=1$ by 
\[
f=\left(\frac{y(y+2)x+y^5}{x+y(y+2)},y\right).
\]
The matrix associated to $f$ is 
\[
M_f=\left(
\begin{array}{cc}
y(y+2) & y^5\\
1 & y(y+2)
\end{array}
\right)
\]
and the Baum Bott index $\mathrm{BB}(f)$ of
$f$ is $\frac{4(y+2)^2}{(y+2)^2-y^5}$. 
In particular $f$ is a 
Jonqui\`eres twist (Proposition \ref{pro:bb}).

According to Theorem \ref{thm:main}.2.b. 
one has $\mu(f)=3$.
An other way to see that is to compute $\deg f^k$ for
any $k$: for any $\ell\geq 1$ one has
\begin{align*}
&\deg(f^{2\ell})=3(\ell+1),  &&\deg(f^{2\ell+1})=3\ell+5.
\end{align*}
Then we find again $\mu(f)=2\displaystyle\lim_{k\to +\infty}\frac{\deg(f^k)}{k}=3$.

\subsubsection{Fourth example}

Consider the map $f$ of $\mathrm{J}$ associated to 
the matrix
\[
M_f=\left(
\begin{array}{cc}
y(y+2)^8 & y^5\\
1 & y(y+2)^8
\end{array}
\right).
\]
The Baum Bott index $\mathrm{BB}(f)$ of $f$ is 
$\frac{4(y+2)^{16}}{(y+2)^{16}-y^3}$ and $f$ belongs to $\mathcal{J}$ 
(Proposition \ref{pro:bb}).
According to Theorem \ref{thm:main}.2.b. 
one has $\mu(f)=16$.
An other way to see that is to compute $\deg f^k$ for
any $k$: for any $k\geq 1$ one has $\deg f^k=8k+2$. 
Then we find again 
$\mu(f)=2\displaystyle\lim_{k\to +\infty}\frac{\deg(f^k)}{k}=2\times 8=16$.

\subsubsection{Fifth example}

Let us consider the Jonqui\`eres map of 
$\mathbb{P}^2_\mathbb{C}$ given in the affine chart $z=1$ by
\[
f=\left(\frac{y(y+1)(y+2)x+y^2}{(y+2)x+y(y+1)(y+2)},y\right).
\]
The matrix associated to $f$ is 
\[
M_f=\left(
\begin{array}{cc}
y(y+1)(y+2) & y^2 \\
y+2 & y(y+1)(y+2)
\end{array}
\right)
\]
and the Baum-Bott index $\mathrm{BB}(f)$ 
of $f$ is $\frac{4(y+1)^2(y+2)}{(y+1)^2(y+2)-1}$; in particular $f$ is a 
Jonqui\`eres twist (Proposition \ref{pro:bb}).

Theorem \ref{thm:main}.2.c. asserts that $\mu(f)=3$.
An other way to see that is to compute $\deg f^k$ for
any $k$: for any $k\geq 1$
\begin{align*}
& \deg(f^{2k})=3k+2 && \deg(f^{2k+1})=3k+4
\end{align*}
so $2\displaystyle\lim_{k\to +\infty}\frac{\deg(f^k)}{k}=3$
and we find again $\mu(f)=3$.

\subsection{Families} 

\subsubsection{First family}

Let us consider the family $(f_t)_t$ of elements of $\mathrm{J}$ 
given by $f_t=\left(x+t,y\,\frac{x}{x+1}\right)$. 
A straightforward computation yields to
\[
f_t^n=\left(x+nt,y\,\frac{x}{x+1}\,\frac{x+t}{x+t+1}\ldots\,\frac{x+(n-1)t}{x+(n-1)t+1}\right)
\]
The birational map $f_t$ belongs to $\mathcal{J}$
if some multiple of $t$ is equal to $1$, and to 
$\mathrm{J}\smallsetminus\mathcal{J}$ otherwise. Furthermore
\begin{itemize}
\item[$\diamond$] if no multiple of $t$ is equal to $1$, 
then $\mu(f_t)=2$ 
$\Big($because $\displaystyle\lim_{k\to +\infty}\frac{\deg f_t^k}{k}=1\Big)$;

\item[$\diamond$] otherwise $\mu(f_t)=0$.
\end{itemize}

\subsubsection{Second family}

Let us recall a result of 
\cite{CantatDesertiXie}: 
let $f$ be any element of $\mathrm{PGL}_3(\mathbb{C})$, or any  
elliptic element of $\mathrm{Bir}(\mathbb{P}^2_\mathbb{C})$ of 
infinite order; 
then $f$ is a limit of pairwise conjugate loxodromic elements 
$($resp. Jonqui\`eres twists$)$ in the 
Cremona group.
Hence there exist families $(f_n)_n$ of birational
self-maps of the complex projective plane such that
\begin{itemize}
\item[$\diamond$] $\mu(f_n)>0$ for any $n\in\mathbb{N}$; 

\item[$\diamond$] $\mu\big(\displaystyle\lim_{n\to +\infty}f_n\big)=0$.
\end{itemize}

\subsubsection{Third family}

Let us recall a construction given in 
\cite{CantatDesertiXie}.
Consider a pencil of cubic curves with nine distinct base points $p_i$ in $\mathbb{P}^2_\mathbb{C}$. 
Given a point~$m$ in $\mathbb{P}^2_\mathbb{C}$, draw the line $(p_1m)$ and denote by $m'$ the third intersection point 
of this line with the cubic of our pencil that contains $m$: 
the map $m\mapsto \sigma_1(m)=m'$ is a birational involution. 
Replacing $p_1$ by $p_2$, we get a second involution and, for a very general pencil, $\sigma_1\circ \sigma_2$ 
is a Halphen twist that preserves our cubic pencil. 
At the opposite range, consider the degenerate cubic pencil, the members of which are the union of a line through the origin and the circle $C=\{x^2+y^2=z^2\}$. Choose $p_1=(1:0:1)$ and $p_2=(0:1:1)$
as our distinguished base points. Then, $\sigma_1\circ \sigma_2$ is a Jonqui\`eres twist preserving the pencil of lines
through the origin; if the plane is parameterized by $(s, t)\mapsto (st, t)$, this Jonqui\`eres twist is conjugate to 
$(s,t)\mapsto \left( s, \frac{(s-1)t+1}{(s^2+1)t+s-1}\right)$. 
Now, if we consider a family of general cubic pencils converging towards this degenerate pencil, we obtain a sequence 
of Halphen twists converging to a Jonqui\`eres twist. 
So there exists a sequence $(f_n)_n$ of birational self-maps of
$\mathbb{P}^2_\mathbb{C}$ whose limit is also a birational
self-map of $\mathbb{P}^2_\mathbb{C}$ and such that
\begin{itemize}
\item[$\diamond$] $\mu(f_n)=0$ for any $n\in\mathbb{N}$; 

\item[$\diamond$] $\mu\big(\displaystyle\lim_{n\to +\infty}f_n\big)>0$.
\end{itemize}

\vspace*{2cm}

\bibliographystyle{alpha}
\bibliography{biblio}

\begin{thebibliography}{CDXar}

\bibitem[ACn02]{AlberichCarraminana}
M.~Alberich-Carrami\~{n}ana.
\newblock {\em Geometry of the plane {C}remona maps}, volume 1769 of {\em
  Lecture Notes in Mathematics}.
\newblock Springer-Verlag, Berlin, 2002.

\bibitem[BC16]{BlancCantat}
J.~Blanc and S.~Cantat.
\newblock Dynamical degrees of birational transformations of projective
  surfaces.
\newblock {\em J. Amer. Math. Soc.}, 29(2):415--471, 2016.

\bibitem[BD05]{BedfordDiller}
E.~Bedford and J.~Diller.
\newblock Energy and invariant measures for birational surface maps.
\newblock {\em Duke Math. J.}, 128(2):331--368, 2005.

\bibitem[BD15]{BlancDeserti}
J.~Blanc and J.~D\'eserti.
\newblock Degree growth of birational maps of the plane.
\newblock {\em Ann. Sc. Norm. Super. Pisa Cl. Sci. (5)}, 14(2):507--533, 2015.

\bibitem[CD12]{CerveauDeserti}
D.~Cerveau and J.~D\'{e}serti.
\newblock Centralisateurs dans le groupe de {J}onqui\`eres.
\newblock {\em Michigan Math. J.}, 61(4):763--783, 2012.

\bibitem[CDXar]{CantatDesertiXie}
S.~Cantat, J.~D\'eserti, and J.~Xie.
\newblock Three chapters on {C}remona groups.
\newblock {\em Indiana Univ. Math. J.}, to appear.

\bibitem[D\'06]{Deserti:compositio}
J.~D\'{e}serti.
\newblock Sur les automorphismes du groupe de {C}remona.
\newblock {\em Compos. Math.}, 142(6):1459--1478, 2006.

\bibitem[DF01]{DillerFavre}
J.~Diller and C.~Favre.
\newblock Dynamics of bimeromorphic maps of surfaces.
\newblock {\em Amer. J. Math.}, 123(6):1135--1169, 2001.

\bibitem[DS05]{DinhSibony}
T.-C. Dinh and N.~Sibony.
\newblock Une borne sup\'{e}rieure pour l'entropie topologique d'une
  application rationnelle.
\newblock {\em Ann. of Math. (2)}, 161(3):1637--1644, 2005.

\end{thebibliography}

\nocite{}

\end{document}